\documentclass[12pt]{article}
\usepackage{amsmath, amssymb, amsfonts, amsthm, amscd, mathrsfs, enumerate, color}

\usepackage{txfonts}

\newtheorem{theorem}{Theorem}[section]
\newtheorem{definition}[theorem]{Definition}
\newtheorem{corollary}[theorem]{Corollary}
\newtheorem{lemma}[theorem]{Lemma}
\newtheorem{proposition}[theorem]{Proposition}
\newtheorem{remark}[theorem]{Remark}
\newtheorem{question}[theorem]{Question}

\newtheorem{conjecture}[theorem]{Conjecture}

\usepackage[all]{xy}

 \def\part{\partial}
 \newcommand{\GL}{\operatorname{GL}}

\newcommand{\SL}{\operatorname{SL}}

  \newcommand{\Ker}{\operatorname{Ker}}
  \newcommand{\im}{\operatorname{Im}}

\newcommand{\beqn}{\begin{equation}}
\newcommand{\eeqn}{\end{equation}}

\newcommand{\bc}{\mathbb{C}}

\newcommand{\bz}{\mathbb{Z}}

 \newcommand{\cx}{\mathcal{X}}
\newcommand{\U}{\mathcal{U}}
\newcommand{\F}{\mathcal{F}}
\newcommand{\X}{\mathcal{X}}
\newcommand{\E}{\mathcal{E}}
\newcommand{\G}{\mathcal{G}}
\newcommand{\calP}{\mathcal{P}}
\newcommand{\V}{\mathcal{V}}
\newcommand{\calS}{\mathcal{S}}
\begin{document}

\title{An approach towards the Koll\'ar-Peskine problem via the
Instanton Moduli Space}

\author{Shrawan Kumar}

\maketitle

\section{Introduction}\label{sec1}

Koll\'ar and Peskine  (cf. [BC, page 278]) asked the following question on
 complete intersections over the field $\bc$ of complex
 numbers. In this note,  the field $\Bbb C$
 is taken as the
 base field. By a variety, we mean a complex quasiprojective (reduced)
 (but not necessarily irreducible) variety.

\begin{question}\label{q1}
 Let $C_t\subset \Bbb P^3$ be a family of
 {\it smooth} curves parameterized by the formal disc $D:=$ spec $R$, where $R$
  is the formal power series ring $\mathbb{C}[[t]]$ in one variable.
 Assume that the general member of the family is a complete intersection.
  Then, is the special member $C_0$ also a
 complete intersection?
\end{question}

 By using a construction due to Serre, the above
 problem is equivalent to the following (cf. [Ku]).

\begin{question}\label{q2}
 Let $V_t$ be a family of rank two vector
 bundles on $\Bbb P^3$.  Assume that the general member of the family is
 a direct sum of line bundles. Then, is the special member $V_0$ also a
 direct sum of line bundles?
 \end{question}

Let us consider the following  slightly weaker version of
the above question.
\begin{question}\label{q3}
 Let $V_t$ be a family of rank two vector
 bundles on $\Bbb P^3$.  Assume that the general member of the family is
 a trivial vector bundle. Then, is the special member $V_0$ also
 a trivial vector bundle?
 \end{question}

In the next section, we show that the above question is equivalent to a
question on the nonexistence of algebraic maps from $\Bbb P^3$ to the infinite
Grassmannian $\X$ associated to the affine $\SL(2)$. Specifically, we have the
following result (cf. Theorem \ref{2.5}):
\begin{theorem} \label{i2.5} Let $X$ be any irreducible projective variety. Then, the
following two conditions are equivalent:

(a)
Any rank-$2$ vector bundle $\F$  on $X\times D$ with trivial determinant,
such that ${\mathcal{F}}_{| X\times D^*}$ is trivial, is itself trivial.

(b) There exists no nonconstant morphism $X\to \X$.
 \end{theorem}

 Thus, Question \ref{q3} is equivalent to the following question
(cf. Question \ref{q4}):

 \begin{question}\label{iq4}
 Does there exist no nonconstant morphism $\Bbb P^3\to \X$?
 \end{question}

  Let $\text{Mor}^d_* (\mathbb{P}^1 ,\mathcal{X})$ denote the set of base point
  preserving
morphisms
from $\Bbb P^1 \to \mathcal{X}$ of degree $d$.
It is a
complex algebraic variety. As we show in Section 3, any morphism $\phi: \Bbb P^3\to \X$
of degree $d$,  preserving the base points,
canonically induces
  a morphism
  \[\hat{\phi}:
 \mathbb{C}^3\backslash\{ 0\}
  \to \text{Mor}^d_* (\mathbb{P}^1 ,\mathcal{X}).\]
Let $\mathcal{M}_d$ be the set of
isomorphism classes of rank two vector bundles $\V$ over $\Bbb P^2$ with trivial
 determinant and with second Chern class $d$ together with a
 trivialization of $\V_{|\Bbb P^1}$.  Then, $\mathcal{M}_d$ has a natural
 variety structure, which will be referred to by the {\it Donaldson
moduli space}. Donaldson showed that there is a natural diffeomorphism
between  $\mathcal{M}_d$ and the instanton moduli space
 $\mathcal{I}_d$ of Yang-Mills $d$-instantons over the flat $\Bbb R^4$ with group
$SU(2)$ modulo based gauge equivalence. As shown by Atiyah, there is a natural
embedding
$$i: \text{Mor}^d_* (\mathbb{P}^1 ,\mathcal{X}) \hookrightarrow \mathcal{M}_d$$
as an open subset (cf. Proposition \ref{3.1}). Thus, the morphism $\hat{\phi}$
gives rise to a morphism (still denoted by) $\hat{\phi}: \mathbb{C}^3\backslash
\{ 0\}
  \to \mathcal{M}_d$. Define an action of $\bc^*$ on  $\mathbb{C}^3\backslash
  \{ 0\}$ by homothecy and on
 $\Bbb P^2$ via:
$$ z\cdot [\lambda,\mu, \nu]=  [z^{-1}\lambda,\mu, \nu].$$
This gives rise to an action of  $\Bbb C^*$ on $\mathcal{M}_d$ via the pull-back
of bundles. Then, the embedding
$\hat{\phi}$ is $\Bbb C^*$-equivariant (cf. Theorem \ref{3.2}).

We would like to make the following conjecture (cf. Conjecture \ref{3.3}).

\begin{conjecture}  \label{i3.3} For any $d>0$, there does not exist
any $\Bbb C^*$-equivariant morphism $\hat{f}: \mathbb{C}^3\backslash \{
0\} \to \mathcal{M}_d.$
\end{conjecture}

Assuming the validity of the above conjecture \ref{i3.3}, we get that there is no
nonconstant  morphism
 $\phi:\mathbb{P}^3 \to \mathcal{X}.$

 Thus, by Theorem \ref{i2.5}, assuming the validity of the above Conjecture \ref{i3.3},
 any rank-$2$ bundle $\F$  on $\mathbb{P}^3 \times D$ with
 trivial determinant,
such that ${\mathcal{F}}_{| \mathbb{P}^3 \times D^*}$ is trivial, is itself
trivial (cf. Corollary \ref{3.4}).

As a generalization of the above, we would like to make the following conjecture
(cf. Conjecture \ref{3.5}).

\begin{conjecture} For any $n \geq 2$, let $\X_n$ be the infinite Grassmannian associated to
the group $G=\SL(n)$, i.e., $\X_n:= \SL(n, K)/\SL(n, R).$ Then,  there
does not exist any
 nonconstant morphism
 $\phi:\mathbb{P}^{n+1}  \to \mathcal{X}_n.$
\end{conjecture}

Finally, in Section 4, we recall an explicit construction of the moduli
space $\mathcal{M}_d$
via the {\it monad} construction and show that the $\bc^*$-action on $\mathcal{M}_d$
takes a relatively simple form (cf. Lemma \ref{4.2}).

\vskip3ex
\noindent
{\bf Acknowledgements.} It is my pleasure to thank N. Mohan Kumar, who brought
to my attention
the Koll\'ar-Peskine problem and with whom I had several very helpful
conversations/correspondences. I also thank J. Koll\'ar for a 
correspondence. This
work was partially supported by the NSF grant DMS-0901239.

\section{Koll\'ar-Peskine problem and infinite Grassmannian}

For more details on the following construction of the infinite Grassmannians, see
[K, Chapter 13].

 Set
 $
 \mathcal G=\SL(2,K),
 \mathcal P= \SL(2,R),
 $
 where $K:=\Bbb C[[t]][t^{-1}]$ denotes the ring of Laurent series in one variable
 and $R$ is the subring $\bc[[t]]$ of power series.
 The ring homomorphism $R\to\Bbb C,\ t\mapsto
 0$, gives rise to a group homomorphism
 $
 \pi:\mathcal P\to \SL(2,\Bbb C).
 $
 Define $\mathcal B=\pi^{-1}(B)$, where $B \subset
 \SL(2,\Bbb C)$ is the Borel subgroup  consisting of the upper
 triangular matrices.  For any $d\geq 0$, define
 $$
 X_d =
 {\bigcup}^d_{n=0}
 \mathcal B\binom{t^n \,\,0}{0 \,\,\,\, t^{-n}}\mathcal P/\mathcal P\subset
 \mathcal G/\mathcal P.
 $$
 Then, $X_d$ admits a natural structure of a projective
 variety and  $\cup_{d \geq 0} X_d=\mathcal G/\mathcal P$.  Moreover, $X_d$ is
 irreducible (of dimension $d$),
 and  $X_d\hookrightarrow X_{d+1}$ is a closed
 embedding.  In
 particular,  $\mathcal{X}:=\mathcal G/\mathcal P$ is a
 projective  ind-variety.

For any integer $d\geq 0$, consider the set $\mathcal{L}_d $ of $R$-submodules $L
\subset K\otimes_\Bbb C V$ such that
\[t^dL_o\subset L \subset t^{-d}L_o,\,\, \text{and dim}_\Bbb C(L/t^dL_o)=2d,\]
where $V:=\Bbb C^2$ and $L_o:=R\otimes V$. Let
\[\mathcal{L}:=\cup_{d\geq 0}\,\mathcal{L}_d.\]
Any element of $\mathcal{L}$ is called an {\it $R$-lattice} in $K\otimes_\Bbb C V$.

The group $\SL(2,K)$ acts canonically on
$K\otimes_\Bbb C V$. Recall the following from [K, Lemma 13.2.14].

\begin{lemma} \label{2.1} The map $g \SL(2,R)\mapsto gL_o$ (for $g\in \SL(2,K)$) induces
 a bijection $\beta: \mathcal{X}\to \mathcal{L}$.
 \end{lemma}

Let $X$ be any irreducible projective variety and let $\mathcal{F}$ be a
rank two vector bundle on $X\times D$ with trivial determinant, where $D:=
\text{spec }R$. Fix a trivialization
of the determinant of $\F$.
  Assume that
$\mathcal{F}_{| X\times D^*}$ is trivial, $D^*$ being the
punctured formal disc $D^*:= \text{spec }K$. Fix a compatible trivialization $\tau$
of $\mathcal{F}_{| X\times D^*}$ (compatible with the trivialization of the
determinant of $\F$).
For any $x\in X$,
  \[
  H^0 (x\times D, \mathcal{F}) \hookrightarrow H^0(x\times D^*,
  \mathcal{F}) \simeq K\otimes_\Bbb C V.
    \]
Thus,
  \[
L_x := H^0 (x\times D,\mathcal{F}) \hookrightarrow
K\otimes_\Bbb C V.
  \]

It can be seen that $L_x$ is an $R$-lattice in $K\otimes_\Bbb C V$. Moreover, the map
$x \mapsto L_x$ provides a morphism $\phi_\F (\tau): X \to \mathcal{X}$ under the identification
of Lemma \ref{2.1} (depending upon the trivialization $\tau$). If we choose a
different compatible trivialization $\tau'$ of the bundle
$\mathcal{F}_{| X\times D^*}$, it
is easy to see that the morphism $\phi_\F(\tau')$ differs from  $\phi_\F(\tau)$
by the left multiplication of an element $g\in \G$, i.e.,
$$\phi_\F(\tau') (x) = g\phi_\F(\tau)(x), \,\,\,\text{for all}\,\, x\in X.$$
(To prove this, observe that any morphism $X \times D^* \to \SL(2, \bc)$ is constant
in the $X$-variable since  $X$ is an irreducible projective
variety by assumption.)

Set $[\phi_\F]$ as the equivalence class of the map $\phi_\F(\tau):X\to
\mathcal{X}$ (for some compatible trivialization
$\tau$), where two maps $X\to \X$ are called equivalent if they differ by
  left multiplication by an element of
$\G$. Thus,  $[\phi_\F]$ does not depend
upon the choice of the compatible trivialization $\tau$ of
$\mathcal{F}_{| X\times D^*}$.

\begin{lemma} \label{2.2}The bundle $\mathcal{F}$
is trivial on $X\times D$ if and only if the map
$[\phi_F]$ is a constant map.
\end{lemma}
\begin{proof}  If $\mathcal{F}$ is trivial on $X\times D$, then
$[\phi_\F]$ is clearly a constant map.  Conversely, assume that
$[\phi_\F]$ is a constant map. Choose a compatible trivialization $\tau$ of
$\mathcal{F}_{| X\times D^*}$ so that
  \[
L_x = L_o , \qquad \forall x\in X.
  \]

  Let $\phi:=\phi_\F(\tau)$. Take a basis $\{ e_1,e_2\}$ of $V$.  This gives rise to
unique sections $\sigma_1(x),\sigma_2(x) \in H^0(x\times D,
\mathcal{F})$ corresponding to the elements $1\otimes e_1$ and
 $1\otimes e_2$ respectively under the map $\phi$. Let $s_1,s_2 \in H^0(X\times D^*,
\mathcal{F})$ be everywhere linearly independent sections
such that
$\sigma_i(x)_{|_{x\times D^*}} = {s_i}_{|x\times D^*}.$

It suffices to  show that $\sigma_1(x),\sigma_2(x)$ are linearly
independent at 0 as well.  Take a small  open subset $U\subset X$  so that the bundle
$\mathcal{F}_{|U\times D}$ is trivial. Fix a compatible trivialization $\tau'$ of
 $\mathcal{F}_{|U\times D}$.
Then,
 the sections $\sigma_i$
can be thought of as maps $U\times D\overset{\widehat{\sigma_i}}
{\longrightarrow} V$ which are linearly independent over any point of
 $U\times D^*$. From this it is easy to see that $\widehat{\sigma_i}$ are linearly
 independent over any point of
 $U\times D$ since the transition matrix over
 $U\times D^*$ with respect to the two trivializations $\tau$ and $\tau'$ of
 $\mathcal{F}_{|U\times D^*}$
 has determinant $1$. Covering $X$ by such small open subsets
 $U$, the lemma is proved.
\end{proof}

As above,  a bundle $\mathcal{F}$  gives rise to a morphism $\phi_\F
: X \to \mathcal{X}$ (unique up to the left multiplication by an element of $\G$).
  Conversely, any morphism $\phi
: X \to \mathcal{X}$ gives rise to a bundle $\mathcal{F}$. Before we can prove
 this, we need the following result.

Let $\mathfrak V:= \mathbb P^1 \times V \to \mathbb P^1 $ be the trivial rank-$2$
vector bundle over $\mathbb P^1 $, where $V$ is the two dimensional
complex vector space $\Bbb C^2$. For any $g\in \G$, define a rank-$2$ locally
free sheaf $\mathfrak V_g$
on $ \mathbb P^1 $ as the sheaf associated to the following presheaf:

For any Zariski open subset $U \subset \mathbb P^1 $, set
$$ \mathfrak V_g (U)= H^0(U, \mathfrak V), \,\,\,\text{if}\, 0\notin U,\,\,\text{and}
$$
$$ \mathfrak V_g (U)= \{\sigma\in H^0(U\setminus \{0\}, \mathfrak V):(\sigma)_0\in
g(R\otimes_\bc V)\},  \,\,\,\text{if}\,\, 0\in U,
$$
where $(\sigma)_0$ denotes the germ of the rational section $\sigma$ at $0$ viewed
canonically as an element of $K\otimes_\bc V$.

With this notation, we have the
 following result from [KNR, Proposition 2.8]. (In fact, we only give
 a particular case of loc. cit. for $G=\SL(2, \bc)$ and for the curve $C= \mathbb P^1$,
 which is sufficient for our purposes.):

\begin{proposition} \label{2.3} There is a rank-$2$ algebraic vector bundle $\U$ on
$\cx \times
\mathbb P^1$ satisfying the following:

(1) The bundle $\U$ is of trivial determinant,

(2)  The bundle  $\U$ is trivial
restricted to
$\X \times (\mathbb P^1 \setminus \{0\})$,

(3)  For any $x= g\calP \in \X$ (for $g\in \G$), the restriction
$\U_{|x\times \mathbb P^1}$ is isomorphic
with the locally free sheaf $\mathfrak V_g$ as above.
\end{proposition}

\begin{lemma}\label{2.4} For any morphism $\phi
: X \to \mathcal{X}$, there exists a
rank two vector bundle $\F_\phi$ on $X\times D$ with trivial determinant
(explicitly constructed in the proof)
  such  that
${\mathcal{F}_\phi}_{| X\times D^*}$ is trivial and such that the associated
morphism $[\phi_{\F_\phi}] = [\phi ]$.
\end{lemma}

\begin{proof} As in  Proposition \ref{2.3}, consider the vector bundle  $\U$ on
$\cx \times
\mathbb P^1$ of rank two. Let $\U_\phi$ be the pull-back of the family $\U$ to $X\times
\Bbb P^1$ via the morphism $\phi \times Id$. Let $\F_\phi$ be the restriction of
$\U_\phi$ to  $X\times D$. Then, by the properties (1)-(2) of Proposition
\ref{2.3}, the bundle $\F_\phi$ satisfies the first two properties of the lemma.
Finally, by the property (3) of Proposition \ref{2.3} and
the definition of the map $[\phi_{\F_\phi}]$, it is easy to see that
$[\phi_{\F_\phi}]
= [\phi ]$.
\end{proof}

Combining  Lemmas \ref{2.2} and \ref{2.4}, we get the following theorem:
\begin{theorem} \label{2.5} Let $X$ be any irreducible projective variety. Then, the
following two conditions are equivalent:

(a)
Any rank-$2$ vector bundle $\F$  on $X\times D$ with trivial determinant,
such that ${\mathcal{F}}_{| X\times D^*}$ is trivial, is itself trivial.

(b) There exists no nonconstant morphism $X\to \X$.
 \end{theorem}

 By virtue of the above theorem, an affirmative answer of Question \ref{q3}
 is equivalent to an affirmative answer of the
 following question. Observe that under the assumptions of Question \ref{q3}, the
 family $V_t$, thought of as a rank-$2$ vector bundle $\V$ on $\Bbb P^3\times D$, has
 trivial determinant by virtue of [H, Exercise 12.6(b), Chap. III]. Also,
 $\V_{|\Bbb P^3\times D^*}$ is trivial by the semicontinuity theorem (cf. [H, $\S$12,
 Chap. III]).

 \begin{question}\label{q4}
 Does there exist no nonconstant morphism $\Bbb P^3\to \X$?
 \end{question}

\begin{definition} \label{2.6}{\rm
Recall (cf. [K, Proposition 13.2.19 and its proof]) that the singular homology
$H_2(\X, \bz) \simeq \bz$ and it has  a canonical generator given by the Schubert
cycle of complex dimension $1$. For any morphism $\phi: \Bbb P^3 \to \X$, define its
{\it degree} to be the integer
 $d=d_\phi$ such that the induced map in homology $\phi_* : H_2(\Bbb P^3, \bz)
 \to H_2(\X, \bz)$ induced by $\phi$ is given via multiplication by $d$.

 Since the pull-back of the ample generator of Pic $\X \simeq H^2(\X, \bz)$ (which
 is globally generated) is
 a globally generated line bundle on $\Bbb P^3$, $d \geq 0$ and $d=0$ if and only
 if $\phi$ is a constant map.

 For any rank-$2$ bundle $\F$  on $\Bbb P^3\times D$
with trivial determinant such that ${\mathcal{F}}_{| X\times D^*}$ is trivial,
we define its}  deformation index $d(\F)=d_{[\phi_\F]}$.
\end{definition}

 \begin{proposition} \label{2.7}  For any morphism $\phi: \Bbb P^3 \to \X$,
 $d_\phi$ is divisible by $6$.

 Equivalently, for any $\F$ as in the above definition, $d(\F)$ is divisible by
 $6$.
 \end{proposition}
 \begin{proof} Consider the induced algebra homomorphism in cohomology:
 $$\phi^*:H^*(\X, \bz) \to H^*(\Bbb P^3, \bz),$$
 induced by $\phi$. By the definition, the induced map at $H^2$ is
 multiplication by $d_\phi$. Moreover, by [K, Exercise 11.3.E.4], for any $i \geq 0$,
 $H^{2i}(\X, \bz)$ is a free $\bz$-module of rank $1$ generated by the Schubert class
 $\epsilon_i$. Moreover,
 $$ \epsilon_i\cdot \epsilon_j = \binom{i+j}{i}\epsilon_{i+j}.$$
 In particular, $6\epsilon_{3}= \epsilon_{1}^3$. From this the
 proposition follows.
 \end{proof}

 \section{Koll\'ar-Peskine problem and the instanton moduli space}

  Take any morphism $\phi: \Bbb P^3 \to \X$, with degree
 $d=d_\phi$. Assume that
$\phi([0,0,0,1])$ is the base point $x_o:=1\cdot \mathcal{P}\in \mathcal{X}$.

Define the map
 \[\pi: \mathbb{C}^3\backslash\{0\} \times\mathbb{P}^1
\longrightarrow \mathbb{P}^3 ,\,\,
(x, [\lambda,\mu ]) \longmapsto [\lambda x, \mu ],\]
for $x\in  \mathbb{C}^3\backslash\{0\}$ and $[\lambda, \mu]\in \mathbb{P}^1$.
There is an action of   $\mathbb{C}^*$
on $\mathbb{C}^3\backslash \{
0\}\times\mathbb{P}^1$ by
  \[
z \cdot\bigl( x, [\lambda ,\mu ]\bigr) = \Bigl( zx, [
\frac{1}{z}\lambda , \mu]\Bigr),\,\,\text{for}\,\, z\in \Bbb C^*.
  \]
Then, $\pi$ factors through the $\Bbb C^*$-orbits. Consider the composite morphism
\[\bar{\phi}=\phi\circ\pi : \mathbb{C}^3\backslash\{ 0\}
  \times\mathbb{P}^1 \to \mathcal{X}.\]
  Observe that $\bar{\phi}(x, 0)=x_o$ for any $x\in  \mathbb{C}^3\backslash\{ 0\}$,
  where $0\in \Bbb P^1$ is the point $[0,1]$.

  Let $\text{Mor}^d_* (\mathbb{P}^1 ,\mathcal{X})$ denote the set of base point
  preserving
morphisms
from $\Bbb P^1 \to \mathcal{X}$ of degree $d$  (taking $0$ to $x_o$). Then, as in
[A, $\S$ 2], $\text{Mor}^d_* (\mathbb{P}^1 ,\mathcal{X})$ acquires the structutre of a
complex algebraic variety.

The map $\bar{\phi}$ canonically induces
  the morphism
  \[\hat{\phi}:
 \mathbb{C}^3\backslash\{ 0\}
  \to \text{Mor}^d_* (\mathbb{P}^1 ,\mathcal{X}).\]

Let us consider the embedding $\Bbb P^1 \hookrightarrow \Bbb P^2$,
$[\lambda, \mu] \mapsto [\lambda, \mu, 0].$  Fix $d\geq 0$ and let $\mathcal{M}_d$ be the set of
isomorphism classes of rank two vector bundles $\V$ over $\Bbb P^2$ with trivial
 determinant and with second Chern class $d$ together with a
 trivialization of $\V_{|\Bbb P^1}$. The isomorphism is required to
preserve
 the trivialization of $\V$ over $\Bbb P^1$. Then, $\mathcal{M}_d$ has a natural
 variety structure. Moreover, any bundle $\V\in  \mathcal{M}_d$ is semistable.
 (By [OSS, Chapter I, Lemma 3.2.2],
   $\V$ is trivial on generic lines $\ell \subset \Bbb P^2$. Thus, by
   [OSS, Chapter II,
   Lemma 2.2.1], $\V$ is semistable.)  We will refer to  $\mathcal{M}_d$ as the {\it Donaldson
moduli space}. Donaldson [D] showed that there is a natural diffeomorphism
between  $\mathcal{M}_d$ and the instanton moduli space
 $\mathcal{I}_d$ of Yang-Mills $d$-instantons over the flat $\Bbb R^4$ with group
$SU(2)$ modulo based gauge equivalence.

Define an action of  $\Bbb C^*$ on
$\text{Mor}^d_* (\mathbb{P}^1 ,\mathcal{X})$ via:
\begin{equation}\label{e1} (z\cdot\gamma)[\lambda,\mu]=\gamma[z\lambda,\mu],
\end{equation}
for $z\in \Bbb C^*$, $\gamma \in \text{Mor}^d_* (\mathbb{P}^1 ,\mathcal{X})$
and $[\lambda,\mu]\in \Bbb P^1$.

Also, define the  action of  $\Bbb C^*$ on  $\Bbb P^2$ via:
\begin{equation}\label{e2} z\cdot [\lambda,\mu, \nu]=  [z^{-1}\lambda,\mu, \nu].
\end{equation}
This gives rise to an action of  $\Bbb C^*$ on $\mathcal{M}_d$ via the pull-back
of bundles, i.e., for $\V\in \mathcal{M}_d, [X]\in \Bbb P^2$, the fiber of $z\cdot \V$ over
$[X]$ is given by:
 \begin{equation}\label{e3} (z\cdot \V)_{[X]}=  \V_{z\cdot [X]}.
\end{equation}(Observe that  $\Bbb P^1 \hookrightarrow \Bbb P^2$ is stable
under $\Bbb C^*$ and hence the trivialization of $\V_{|\Bbb P^1}$ pulls back to a
trivialization.)

Recall the following result from [A, $\S$ 2].
\begin{proposition} \label{3.1} There is a natural embedding
$$i: \text{Mor}^d_* (\mathbb{P}^1 ,\mathcal{X}) \hookrightarrow \mathcal{M}_d$$
as an open subset. Moreover, $i$ is $\Bbb C^*$-equivariant with respect to the
$\Bbb C^*$ actions as in equations \eqref{e1} and \eqref{e3}.
\end{proposition}

The following result summarizes the above discussion.

\begin{theorem} \label{3.2} To any morphism $\phi:\mathbb{P}^3 \to \mathcal{X}$
 of degree $d$ preserving the base points,
there is a canonically associated  $\Bbb C^*$-equivariant morphism (defined above)
\[\hat{\phi}:
\mathbb{C}^3\backslash \{
0\} \to \mathcal{M}_d,\]
where  $\Bbb C^*$ acts on $\mathbb{C}^3\backslash \{
0\}$ via the multiplication.

Moreover, $\phi$ is constant (i.e., $d=0$) iff $\hat{\phi}$ is constant.
\end{theorem}

We would like to make the following conjecture.

\begin{conjecture} \label{3.3} For any $d>0$, there does not exist
any $\Bbb C^*$-equivariant morphism $\hat{f}: \mathbb{C}^3\backslash \{
0\} \to \mathcal{M}_d.$
\end{conjecture}

Assuming the validity of the above conjecture, we get the following.
\begin{corollary}\label{3.4}  Assuming the validity of Conjecture \ref{3.3}, there
does not exist any
 nonconstant morphism
 $\phi:\mathbb{P}^3 \to \mathcal{X}.$

 Thus, by Theorem \ref{2.5}, assuming the validity of Conjecture \ref{3.3},
 any rank-$2$ bundle $\F$  on $\mathbb{P}^3 \times D$ with
 trivial determinant,
such that ${\mathcal{F}}_{| \mathbb{P}^3 \times D^*}$ is trivial, is itself
trivial.
 \end{corollary}

As a generalization of the above corollary, I would like to make the following conjecture.

\begin{conjecture} \label{3.5} For any $n \geq 2$, let $\X_n$ be the infinite Grassmannian associated to
the group $G=\SL(n)$, i.e., $\X_n:= \SL(n, K)/\SL(n, R).$ Then,  there
does not exist any
 nonconstant morphism
 $\phi:\mathbb{P}^{n+1}  \to \mathcal{X}_n.$
\end{conjecture}

\begin{remark} {\rm An
interesting aspect of this approach is that Question \ref{q3} involving an
arbitrary family of (not necessarily semistable) vector bundles on $\Bbb
P^3$ is reduced to a question about the Donaldson moduli space $
\mathcal{M}_d$ consisting of rank two } semistable
 bundles {\rm on}  $\Bbb P^2$.
 \end{remark}

\section{Monad construction of $\mathcal{M}_d$}

This section recalls an explicit construction of the moduli space $\mathcal{M}_d$
via the monad construction. We refer to [OSS, $\S\S$ 3,4, Chap. II] for more details
on the monad construction (see also [B] and [Hu]).

Fix an integer $d \geq 0$. Let $H,K,L$ be complex vector spaces of dimensions
$d,2d + 2,d$ respectively.
By {\it monad} one means linear maps parameterized by $Z\in \bc^3$, depending
linearly on $Z$:
$$H \overset{A_Z}\to K \overset{B_Z}\to L,$$
such that the composite $B_Z\circ  A_Z=0$, for all  $Z\in \bc^3$. The monad is
said to be {\it nondegenerate} if for all $Z\in \bc^3 \setminus \{0\}$, $B_Z$ is
surjective and $A_Z$ is injective. In this case, we get a vector bundle on
$\Bbb P^2$ with fiber at the line $[Z]$ the vector space
$$\E(A,B):= \Ker B_Z/\im A_Z.$$
Then, any rank-$2$ bundle on $\Bbb P^2$ with the second Chern class  $d$, which is
 trivial on some line,  is isomorphic with $\E(A,B)$, for some monad $(A,B)$.
 Moreover, such a  monad $(A,B)$ is unique
 up to the action of $\GL(H)\times \GL(K)\times \GL(L)$. Let $[\lambda, \mu, \nu]$
 be the homogeneous coordinates on $\Bbb P^2$. If we only consider bundles
 on $\Bbb P^2$  trivial on the fixed line $\nu = 0$, the condition on the
 corresponding
 monad is that the composite $B_\lambda A_\mu =- B_\mu A_\lambda$ is an isomorphism,
 where (for $Z=(\lambda, \mu, \nu)$)
 $$A_Z:=A_\lambda\lambda + A_\mu \mu + A_\nu \nu, \,\,\,\text{and}\,\,
 B_Z:=B_\lambda\lambda + B_\mu \mu + B_\nu \nu. $$
 In the following, $t$ denotes the transpose, $I_{d\times d}$ denotes the identity
 matrix of
 size $d\times d$, $0_{d\times d}$ denotes the zero matrix of size $d\times d$ and
 $\alpha, \beta, a $ and $b$ are matrices of indicated sizes. For such bundles,
 using the action of $\GL(H)\times \GL(K)\times \GL(L)$, one can choose bases for $H,K,L$  so that the maps are given as
 follows.
$$A_\lambda=\bigl(I_{d\times d}, 0_{d\times d}, 0_{d\times 2}\bigr)^t, A_\mu=
\bigl(0_{d\times d}, I_{d\times d},  0_{d\times 2}\bigr)^t,
A_\nu= \bigl(\alpha_{d\times d}^t, \beta_{d\times d}^t, a_{d\times 2}\bigr)^t,$$
$$B_\lambda=\bigl(0_{d\times d}, I_{d\times d}, 0_{d\times 2}\bigr), B_\mu=
\bigl(-I_{d\times d}, 0_{d\times d},  0_{d\times 2}\bigr),
B_\nu= \bigl(-\beta_{d\times d}, \alpha_{d\times d}, b_{d\times 2}\bigr),$$
and the following condition is satisfied:
$$B_\nu A_\nu =  0,\,\,\,\text{which is equivalent to the condition}\,\,
[\alpha,\beta]+ b a^t=0.$$
The restriction of the bundle $\E(A,B)$ to the line $\nu =0$ has a standard frame
given by the last $2$ basis vectors of $K\simeq \bc^{2d+2}$.

For any $d\geq 0$, let ${\hat{\calS}}_d$ be the closed subvariety  of matrices
$(\alpha, \beta, a, b)$
such that $\alpha, \beta$ are $d \times d$ matrices and  $a, b$ are
 $d \times 2$
 matrices and they  satisfy:

(1) $[\alpha,\beta]+ b a^t=0$.

\vskip2ex

Let $\calS_d$ be the open subset of ${\hat{\calS}}_d$ satisfying, in addition, the
following condition:

\vskip2ex
(2) For all $\lambda, \mu \in \bc$, $\bigl(\alpha^t+\lambda I_{d\times d},
\beta^t +\mu I_{d \times d}, a\bigr)^t$ is injective and
$\bigl(-(\beta+ \mu I_{d\times d}),
\alpha +\lambda I_{d \times d}, b\bigr)$ is surjective.

\vskip2ex
We recall the following result due to Barth from [D, Proposition 1].
\begin{theorem} For any $d\geq 0$, the variety $\mathcal{M}_d$ is isomorphic with
the quotient of the variety $\calS_d$ by the action of $\GL(d)$ under:
$$g\cdot (\alpha, \beta, a, b)= (g\alpha g^{-1}, g\beta g^{-1}, (g^{-1})^t a, gb),$$
for $g\in \GL(d)$, and $(\alpha, \beta, a, b)\in \calS_d$.
\end{theorem}
\begin{remark} {\rm The affine variety ${\hat{\calS}}_d$ is stable under the above action of
$\GL(d)$. Moreover, the open subset of stable points  of
${\hat{\calS}}_d$ (under the $\GL(d)$-action) is precisely equal to $\calS_d$ (cf.
[D, Lemma on page 458 and its proof]).}
\end{remark}
\begin{lemma}\label{4.2} Under the above isomorphism of  the variety $\mathcal{M}_d$ with
the quotient of $\calS_d$ by $\GL(d)$, the action of $\bc^*$ transports to the action:
$$ z\cdot (\alpha, \beta, a, b)= (z\alpha, \beta, za, b),\,\,\,\text{for}\,\,
z\in \bc^*, \,(\alpha, \beta, a, b)\in \calS_d.$$
\end{lemma}
\begin{proof} The  $\bc^*$-action on $\mathcal{M}_d$ via the pull-back corresponds to
the bundle:
\begin{align*}
\frac{\Ker (z^{-1}\lambda B_\lambda + \mu B_\mu + \nu B_\nu)}{\im
(z^{-1}\lambda A_\lambda + \mu A_\mu + \nu A_\nu)}&= \frac{\Ker (-\mu I_{d\times d}
-\nu \beta,
z^{-1}\lambda I_{d\times d}+ \nu \alpha, \nu b )}{\im
\bigl((z^{-1}\lambda I_{d\times d}+ \nu \alpha^t, \mu I_{d\times d}+ \nu \beta^t, \nu a)^t\bigr)}\\
&=\frac{\Ker (-\mu I_{d\times d}-\nu \beta,
z^{-1}\lambda I_{d\times d}+ \nu \alpha, \nu b )}{\im
\bigl((\lambda I_{d\times d}+ z\nu \alpha^t, z\mu I_{d\times d}+ z\nu \beta^t, z\nu a)^t\bigr)}.
\end{align*}
Changing the basis in $\bc^{2d+2}= \bc^d\times \bc^d\times \bc^2$ in the second
factor to $\{ze_j\}_{1\leq j\leq d}$, where $\{e_j\}$ is the original basis, we
get that the last term in the above equation is equal to
$$\frac{\Ker (-\mu I_{d\times d}-\nu \beta,
\lambda I_{d\times d}+ z\nu \alpha, \nu b )}{\im
\bigl((\lambda I_{d\times d}+ z\nu \alpha^t, \mu I_{d\times d}+ \nu \beta^t, z\nu a)^t\bigr)}.$$
This proves the lemma.
\end{proof}

Address: Department of Mathematics, University of North Carolina, Chapel Hill, NC
27599-3250, USA (email: shrawan$@$email.unc.edu)
\end{document}